\documentclass[11pt]{article}

\usepackage[letterpaper, hmargin=1in, top=1in, bottom=1.1in, footskip=0.6in]{geometry}

\usepackage{titlesec}
\titleformat{\subsection}[runin]{\normalfont\bfseries}{\thesubsection.}{.5em}{}[.]\titlespacing{\subsection}{0pt}{2ex plus .1ex minus .2ex}{.8em}
\titleformat{\subsubsection}[runin]{\normalfont\itshape}{\thesubsubsection.}{.3em}{}[.]\titlespacing{\subsubsection}{0pt}{1ex plus .1ex minus .2ex}{.5em}
\titleformat{\paragraph}[runin]{\normalfont\itshape}{\theparagraph.}{.3em}{}[.]\titlespacing{\paragraph}{0pt}{1ex plus .1ex minus .2ex}{.5em}

\usepackage{amsmath}
\usepackage{amssymb}
\usepackage{amsfonts}
\usepackage{latexsym}
\usepackage{amsthm}
\usepackage{amsxtra}
\usepackage{amscd}
\usepackage{bbm}
\usepackage{mathrsfs}
\usepackage{bm}
\usepackage{tensor}

\usepackage{graphicx, color}

\definecolor{darkred}{rgb}{0.9,0,0.3}
\definecolor{darkblue}{rgb}{0,0.3,0.9}

\definecolor{vdarkred}{rgb}{0.7,0,0.2}
\definecolor{vdarkblue}{rgb}{0,0.2,0.7}
\usepackage[pdftex, colorlinks, linkcolor=vdarkblue,urlcolor=black,citecolor=vdarkred]{hyperref}

\usepackage[nottoc,notlof,notlot]{tocbibind}
\usepackage{cite}

\numberwithin{equation}{section}
\numberwithin{figure}{section}

\theoremstyle{plain} 
\newtheorem{theorem}{Theorem}[section]
\newtheorem*{theorem*}{Theorem}
\newtheorem{lemma}[theorem]{Lemma}
\newtheorem*{lemma*}{Lemma}

\newtheorem*{corollary*}{Corollary}

\newtheorem*{proposition*}{Proposition}

\newtheorem*{conjecture*}{Conjecture}

\theoremstyle{definition} 

\newtheorem*{definition*}{Definition}

\newtheorem*{example*}{Example}

\newtheorem*{remark*}{Remark}

\newtheorem*{assumption*}{Assumption}

\renewcommand{\b}[1]{\boldsymbol{\mathrm{#1}}} 
\newcommand{\bb}{\mathbb} 
\renewcommand{\cal}{\mathcal}

\newcommand{\Q}{\mathbb{Q}}
\newcommand{\e}{\mathrm{e}}

\newcommand{\ii}{\mathrm{i}}

\newcommand*{\deq}{\mathrel{\vcenter{\baselineskip0.65ex \lineskiplimit0pt \hbox{.}\hbox{.}}}=}

\renewcommand{\leq}{\leqslant}

\renewcommand{\geq}{\geqslant}

\renewcommand{\epsilon}{\varepsilon}

\newcommand{\qq}[1]{[\![{#1}]\!]}

\DeclareMathOperator{\diag}{diag}
\DeclareMathOperator{\tr}{Tr}

\DeclareMathOperator{\re}{Re}

\newcommand*{\rom}[1]{\expandafter\@slowromancap\romannumeral #1@}

\title{\bf \Large On eigenvectors of the SYK model \vspace{0.5em}}

\author{Yukun He\footnote{Shanghai Center for Mathematical Sciences and School of Mathematical Sciences, Fudan University. Email: \href{mailto:heyukun@fudan.edu.cn}{heyukun@fudan.edu.cn}. }\vspace{1em}}

\begin{document}
\maketitle

\begin{abstract}
	We consider the Sachdev--Ye--Kitaev model with $N$ Majorana fermions and interaction order $q$. For even $q\in \{2,4,...,N-2\}$, we give an elementary proof and show that the eigenvectors of the model are completely delocalized in all deterministic directions. 
\end{abstract}

\section{Introduction}
Throughout the paper, let $N,q$ be even integers such that $2\leq q\leq N-2$. We abbreviate $\qq{N}\deq \{1,2,...,N\}$ and $L\deq2^{N/2}$. Consider the Hamiltonian
\[
H=\sum_{A\subset \qq{N}, |A|=q} J_A\Psi_A\,,
\] 
where $J_A$ are independent standard Gaussian random variables, and $\Psi_A \in \bb C^{L\times L}$ are deterministic and satisfy
\begin{equation} \label{1.1}
	\Psi_A^2=I\,,\quad\Psi_A \Psi_B =(-1)^{|A\cap B|}\Psi_B \Psi_A\,.
\end{equation}
The explicit form of $\Psi_A$ is constructed in Section  \ref{sec2}. This (up to a normalization factor) is the Sachdev--Ye--Kitaev (SYK) model
\cite{SY93,K15}, a quantum many-body system that has become one of the
most studied models in theoretical physics over the past decade. There has been many advances in understanding the spectrum and free energy of the model \cite{BKM26,BC26,FengTianWei1,FengTianWei2,FengTianWei3,GPSZ26,GJV18,GGV16,GGV17,H26,JV18,PR16}. For more details, please refer to \cite{H26} and the references within. 

In this article, we study the eigenvectors of the SYK model. It is a fundamental question for any many-body system that whether it is chaotic (usually with delocalized eigenvectors), or regular (usually with localized eigenvectors). The literature on this aspect of the SYK model is very limited \cite{HG19,SC17}. The recent paper \cite{BC26} proves that for fixed $q\geq 4$, the bulk eigenvectors of $H$ have the $\ell^\infty$ bound
\[
N^{-1/4+o(1)}\asymp(\log L)^{-1/4+o(1)}
\]
in any deterministic directions. Remark 1.12 of \cite{BC26} also mentions that Gaussian rotation invariance of the SYK model allows one to prove complete delocalization in occupation basis, but this restriction to basis is essential to that argument.

In this paper, we prove delocalization of the SYK model in full generality. We may now state our main result.

\begin{theorem} \label{thm1.1}
	Let $(\b u_\alpha)_{\alpha=1}^L$ be the unit eigenvectors of $H$. Let $\b w\in \bb C^L$ be a deterministic unit vector. Then for any $c>0$,
	\begin{equation} \label{thmbound}
		\bb P\big(	\max_{\alpha} |\langle \b u_{\alpha}, \b w\rangle |\geq L^{-1/2}\e^{c\sqrt{N \log N}} \big)\leq \exp\Big(-\Big(4c^2-\frac{\log 2}{2}+o(1)\Big)N\Big)\,.
	\end{equation}
\end{theorem}

Theorem \ref{thm1.1} shows that for all interaction orders $2\leq q\leq N-2$, all eigenvectors of $H$ are completely delocalized in all deterministic directions, with the near-optimal $\ell^\infty$ bound
\[
L^{-1/2+O(\sqrt{\log N/N})}\,.
\] 
Interestingly, the SYK system is integrable when $q=2$, yet its eigenvectors are nevertheless delocalized.

For the proof of Theorem \ref{thm1.1}, recall that for the Gaussian Unitary Ensemble, as the model is unitary invariant, much of its eigenvector information is immediate. The SYK model does not enjoy this full invariance. However, by examining the structure of $\Psi$ and the Gaussianity of $J$, we show that the model is invariant under a special class of unitary matrices, which are introduced in Section \ref{sec2} as spin rotations. We further construct a random spin rotation $\cal U$, such that every spectral projector of $H$ satisfies
\[
P_\lambda(H)\overset{d}{=}P_\lambda(\cal UH\cal U^*)=\cal U P_\lambda(H)\cal U^*\,.
\] 
In addition, the randomness of $\cal U$ allows us to estimate arbitrary moments of $\langle \b w , \cal UP_\lambda \cal U^* \b w \rangle $ for any deterministic $\b w$. The resulting bound has a dependency on eigenvalue multiplicity, and we conclude the proof by showing that the eigenspaces of $H$ have rank at most two almost surely.

\subsection*{Acknowledgment}
Supported by National Key R\&D Program of China
No.\,2023YFA1010400 and NSFC No.\,12322121.

\section{Explicit Majorana fermions} \label{sec2}

We recall the constructions of Majorana fermions $\psi_1,...,\psi_{N+1}\in\mathbb C^{L\times L}$, which are deterministic Hermitian matrices satisfying the relation
\[
\psi_i\psi_j+\psi_j\psi_i=2\delta_{ij}\,.
\]
Having $\{\psi_i\}_{i=1}^{N+1}$ at hand, for $A=\{i_1,...,i_q\}\subset\qq{N}$, $i_1<i_2<\cdots<i_q$, one can simply set
\[
\Psi_A=\ii^{q/2}\psi_{i_1}\cdots\psi_{i_q}\footnote{In this construction, we make use of the operators $\psi_1,...,\psi_N$, but one can choose any $N$ of the $N+1$ Majorana fermions.}\,,
\]
which are Hermitian and satisfy \eqref{1.1}.

The construction is inductive in dimension. For each $m\in \bb N$, it suffices to find $\psi_1^{(m)},...,\psi_{2m+1}^{(m)}\in \mathbb C^{2^m\times 2^m}$ with $\psi^{(m)}_i\psi^{(m)}_j+\psi^{(m)}_j\psi^{(m)}_i=2\delta_{ij}$. For $m=0$, simply choose $\psi_1^{(0)}=1$. The induction step is 
\begin{equation} \label{2.1}
	\psi_i^{(m)}=\begin{pmatrix}
		0  &-\ii \psi_i^{(m-1)}\\
		\ii \psi_i^{(m-1)} &0
	\end{pmatrix} \quad \mbox{for}\  1\leq i\leq 2m-1\,, 
\end{equation}
and
\begin{equation} \label{psi^m}
	\psi_{2m}^{(m)}=\begin{pmatrix}
		0 &I_{2^{m-1}}\\
		I_{2^{m-1}} &0
	\end{pmatrix}\,, \quad \psi_{2m+1}^{(m)}=\begin{pmatrix}
		I_{2^{m-1}} &0\\
		0 & -I_{2^{m-1}}
	\end{pmatrix}\,.
\end{equation}

\subsection{Law invariance of \texorpdfstring{$H$}{H}}
Let $m\geq 1$. For $\b a\in \bb R^{k}$, $1\leq k\leq 2m+1$ we write
\begin{equation} \label{gammaa}
	\psi^{(m)}(\b a)=\sum_{i=1}^{k} a_i\psi^{(m)}_i\,.
\end{equation}
It is easy to see that for $|\b a|=1$
\begin{equation} \label{2.4}
	\psi^{(m)}(\b a)^*=\psi^{(m)}(\b a)\,, \quad \psi^{(m)}(\b a)^2=I\,, \quad \psi^{(m)}(\b a)\psi^{(m)}(\b z) \psi^{(m)} (\b a)=\psi^{(m)}(2\langle \b a,\b z\rangle \b a-\b z)
\end{equation}
for all $\b z\in \bb R^k$.

In particular, for $|\b a|=1$, $k=2m=N$ and $\psi_i^{(m)}=\psi_i$, we have
\[
\psi(\b a)\Psi_A \psi(\b a)=\Big(1-2\sum_{i \in A}a_i^2\Big)\Psi_A-2\sum_{i \in A, j\notin A} a_ia_j \Psi_A \psi_i \psi_j\,,
\]
where $\Psi_A \psi_i \psi_j\in \{\pm \Psi_{A\triangle \{i,j\}}\}$. Note that $\tr (\Psi_A \Psi_B)=L\delta_{AB}$. Thus the map $T(X)=\psi(\b a)X\psi(\b a)$ preserves the real vector space generated by $\{\Psi_A: A\subset\qq{N},|A|=q\}$ as well as the inner product $\langle X,Y \rangle\deq L^{-1}\tr (XY)$.  As $J_A$ are i.i.d. Gaussian random variables, it is not hard to see that
\begin{equation}\label{invariance}
		H\overset d=UHU^*
	\end{equation}
	for every finite product $U$ of matrices $\psi(\b a)$, $\b a\in \bb S^{N-1}$.
	
	In the squeal, we call a finite product of $\psi^{(m)}(\b a)$, $\b a\in \bb S^{2m-1}$ a spin rotation of type $m$; a product of an even number of matrices $\psi^{(m)}(\b a)$, $\b a\in \bb S^{2m}$ is called a general spin rotation of type $m$.

\subsection{A random unitary and its associated bound}
Next we construct a particular (random) unitary $\mathcal U$ satisfying \eqref{invariance}, which will be used in the main argument of the proof. In practice, we recursively construct, for each $m\in\mathbb N$, a general spin rotation of type $m$, $U_m$. In the end, set
\begin{equation}\label{calU}
\mathcal U=\diag(U_{N/2-1},U_{N/2-1})\,.
\end{equation}
To see that $\cal U$ is indeed a spin rotation of type $N/2$, simply note that for $\b a,\b b\in \bb S^{N-2}$,
\[
\diag (\psi^{(N/2-1)}(\b a)\psi^{(N/2-1)}(\b b), \psi^{(N/2-1)}(\b a)\psi^{(N/2-1)}(\b b))=\psi ((\b a,0))\psi ((\b b,0))\,.
\]

\begin{lemma}\label{moment}
For every $m\geq0$, there is a general spin rotation of type $m$, $U_m$, such that the following holds. For any deterministic $A\in\mathbb C^{2^m\times2^m}$, set $\mu_m(A)\deq2^{-m}\|A\|_{\mathrm{HS}}^2$. We have
\begin{equation}\label{momentbound}
\mathbb E|\tr(AU_m)|^{2k}
\leq \mu_m(A)^k
\big[(1+4m/k)\big]^{k(k-1)/2}
\end{equation}
for all integers $k\geq2$.
\end{lemma}

\begin{proof}
Set $U_0=1$. For $m\geq1$, abbreviate 
\[
\gamma_i\equiv \psi_i^{(m)}\,,\quad \mbox{and} \quad \eta_i\equiv \psi_i^{(m-1)}\,,
\] 
and the notation \eqref{gammaa} also extends to $\gamma$ and $\eta$. Let $\b x=(x_1,...,x_{2m+1})$ and $\b y=(y_1,...,y_{2m})$ be independent random vectors, uniformly distributed on $\mathbb S^{2m}$ and $\mathbb S^{2m-1}$ respectively\footnote{Implicitly, in this recursive generation, any new random variables are assumed to be independent of the history, and the SYK model $H$. In particular, here $\b x,\b y$ are independent from $U_{m-1}$.}. Append a zero coordinate to $\b y$, and set
\[
\b a=\frac{\b x+e_{2m+1}}{|\b x+e_{2m+1}|},\qquad
\b b=\frac{\b y+e_{2m}}{|\b y+e_{2m}|},\qquad
U_m=\gamma(\b a)\gamma(\b b)\diag(U_{m-1},U_{m-1}).
\]
Here the probability that either denominator vanishes is $0$. 

We shall prove that
\begin{equation}\label{2.8}
\mathbb E|\tr(AU_m)|^{2k}\leq D_{m,k}\mu_m(A)^k,\quad
D_{m,k}\deq\prod_{r=1}^{2m}\prod_{i=1}^{k-1}\Big(1+\frac{i}{r+i}\Big)\,.
\end{equation}
The estimate
\[
\begin{aligned}
	\log D_{m,k}
	\leq\sum_{i=1}^{k-1}i\int_0^{2m}\frac{dx}{i+x}
	=\sum_{i=1}^{k-1}i\log(1+2m/i)
	\leq\frac{k(k-1)}2\log(1+4m/k)
\end{aligned}
\]
then gives the desired result. Here in the third step we used Jensen's inequality.

The proof of \eqref{2.8} again follows from an induction on $m$. As $U_0=1$ and $D_{0,k}=1$, the statement is clearly true for $m=0$. Assume the result at $m-1$. Set $\b y'=(y_1,\ldots,y_{2m-1})$ and 
the unitary matrix
\[
Q\deq\frac{(1+y_{2m})I-\ii\eta(\b y')}{\sqrt{2(1+y_{2m})}}\,.
\] 
By \eqref{2.1}, \eqref{psi^m} and the last relation of \eqref{2.4}, we have
\[
\gamma(\b a)\gamma_{2m+1}\gamma(\b a)=\gamma(\b x),\qquad
\gamma(\b b)=\begin{pmatrix}0&Q\\Q^*&0\end{pmatrix},\qquad
(Q^*)^2=y_{2m}I+\ii\eta(\b y')\,.
\]
Let $A\ne0$, and put $T=A\gamma(\b a)$. Use $T_{ij}$ to denote the four $2^{m-1}\times2^{m-1}$ blocks of $T$. Define $B=T_{12}Q^*+T_{21}Q$. It is not hard to see that $
\tr(AU_m)=\tr(BU_{m-1}).
$
As a result,
\begin{equation} \label{ind}
\mathbb E\big(|\tr(AU_m)|^{2k}\mid\b x,\b y\big)
\leq D_{m-1,k}\mu_{m-1}(B)^k.
\end{equation}
Thus the next natural target is relating $\mu_m(A)$ to $\mu_{m-1}(B)$. 

To this end, set $v_{\b x}=\|T_{12}\|_{\textrm{HS}}^2+\|T_{21}\|_{\textrm{HS}}^2$. We have
\[
2v_{\b x}=\|T\|_{\textrm{HS}}^2-\tr (T^*\psi_{2m+1}^{(m)}T\psi_{2m+1}^{(m)})=\|A\|_{\textrm{HS}}^2-\tr (A^*\psi_{2m+1}^{(m)}A\gamma(\b x))
\]
Set $(\b b_A)_i=-\tr (A^*\psi_{2m+1}^{(m)}A\gamma_i)/\|A\|_{\textrm{HS}}^2$. Then
\begin{equation} \label{2.9}
	\frac{2v_{\b x}}{\|A\|_{\textrm{HS}}^2}=1+\b b_A \cdot \b x\,.
\end{equation}
Since the LHS of the above is nonnegative and $|\b x|=1$, we have $|\b b_A|\leq 1$. Similarly, 
\begin{equation} \label{2.10}
	\|B\|_{\textrm{HS}}^2=v_{\b x}+2\re \tr (T_{21}^*T_{12}(Q^*)^2)=v_{\b x}+2\re \tr (T_{21}^*T_{12}(y_{2m}I_{2^{m-1}}+\ii \eta (\b y')))\,.
\end{equation}
Now, if $v_{\b x}>0$, similar to \eqref{2.9}, we get from \eqref{2.10} that $\|B\|_{\textrm{HS}}^2/v_{\b x}=1+\b c_{A,\b x}\cdot \b y$ for some $|\b c_{A,\b x}|\leq 1$. If $ v_{\b x}=0$, we get from the definition of $B$ that $B=0$, and we set $\b c_{A,\b x}=\b 0$. In both cases of $v_{\b x}$, we obtained that
\begin{equation} \label{2.12}
\mu_{m-1}(B)=\mu_m(A)(1+\b b_A \cdot \b x)(1+\b c_{A,\b x}\cdot \b y)\,,\quad |\b b_A|\,,|\b c_{A,\b x}|\leq 1\,.
\end{equation}

Note that spherical symmetry gives
\[
\mathbb E(1+\b b_A \cdot \b x)^k
=\mathbb E(1+|\b b _A|x_1)^k
\leq\mathbb E(1+x_1)^k
=\prod_{i=1}^{k-1}\Big(1+\frac{i}{2m+i}\Big)\,,
\]
where the second step follows by expanding in even powers of $|\b b _A|$. The analogue also holds for moments of $1+\b c_{A,\b x}\cdot \b y$. Thus \eqref{ind} and \eqref{2.12} imply
\[
\begin{aligned}
\mathbb E|\tr(AU_m)|^{2k}
\leq D_{m-1,k}\mu_m(A)^k\prod_{i=1}^{k-1}\Big(1+\frac{i}{2m+i}\Big)\Big(1+\frac{i}{2m-1+i}\Big)
=D_{m,k}\mu_m(A)^k.
\end{aligned}
\]
This proves \eqref{2.8} and thus finishes the proof.
\end{proof}

\section{Proof of the main result}

With the help of $\mathcal U$ and the related estimates in Section \ref{sec2}, we can now prove Theorem \ref{thm1.1}, with the aid of the following bound on the eigenvalue multiplicity.

\begin{lemma}\label{multiplicity}
The eigenspaces of $H$ have dimension at most two almost surely.
\end{lemma}

\begin{proof}
Write $m=N/2$, $r=q/2$, and $\Gamma_i=\ii\psi_{2i-1}\psi_{2i}$ for all $i=1,2,...,m$. We see that the $\Gamma_i$ form a family of commuting unitary Hermitian matrices, and thus they have a common orthonormal eigenbasis, with eigenvalues $\pm1$. Enumerate the $r$-element subsets of $\qq{m}$ as $I_1,...,I_M$, where $M=\binom{m}{r}$. Write 
\[
H_0=\sum_{\ell=1}^M 2^{\ell} \prod_{i \in I_{\ell}} \Gamma_i \quad \mbox{with eigenvalues}\quad  \lambda_{\b z} = \sum_{\ell=1}^M 2^{\ell} \prod_{i \in I_{\ell}} z_i\,, \quad \b z\in \{-1,1\}^{m}\,.
\]
Note that for $\sum_{\ell=1}^{M}$, each power exceeds the sum of the smaller powers. Thus, setting $h_i=z_i z_i'$, we have
\begin{equation} \label{3.5}
	\lambda_{\b z}=\lambda_{\b z'} \quad \mbox{iff} \quad \prod_{i\in I_{\ell}}h_i=1\,, \ \mbox{for all}\ \ell\in \qq{M}\,.
\end{equation}
For each $i\neq j$, applying the RHS of \eqref{3.5} to $I_{\ell_1}$ and $I_{\ell_2}$ with $I_{\ell_1}\triangle I_{\ell_2}=\{i,j\}$, we see that $h_i=h_j$. Thus $\lambda_{\b z}=\lambda_{\b z'}$ implies $\b z=h\b z'$ for some $h\in \{-1,1\}$. The relation $\prod_{i\in I_{\ell}}h_i=1$ further shows $h^r=1$.  As a result, $H_0$ has multiplicities at most 2; when $r$ is odd, $H_0$ has simple spectrum.

Note that $H_0$ is a one-point realization of $H$. Since $H_0$ has multiplicities at most two, the set of coefficients $(J_A)$ for which $H$ has an eigenvalue of multiplicity at least three has Lebesgue measure zero (in dimension {$N\choose q$}) . As $J_A$ are independent and absolutely continuous, we finish the proof.
\end{proof}

\begin{proof}[Proof of Theorem \ref{thm1.1}]
Let a deterministic unit vector $\b w\in\mathbb C^L$ be given. Let $\mathcal U$ be as in \eqref{calU}, and independent of $H$. By \eqref{invariance}, we have $H\overset d=\mathcal U H\mathcal U^*$. For a unit vector $\b u\in\mathbb C^L$, write $\b w_+,\b w_-$ and $\b u_+,\b u_-$ for the two blocks of $\b w$ and $\b u$ in $\bb C^{L/2}$. Set $A\deq\b u_+\b w_+^*+\b u_-\b w_-^*$. We have
\[
\langle\b w,\mathcal U\b u\rangle=\tr(AU_{N/2-1})\,.
\]
As $\|A\|_{\mathrm{HS}}\leq|\b u_+||\b w_+|+|\b u_-||\b w_-|\leq1$, we get $\mu_{N/2-1}(A)\leq2/L$. Lemma \ref{moment} then gives
\begin{equation} \label{lll}
\mathbb E|\langle\b w,\mathcal U\b u\rangle|^{2k}
\leq (2/L)^k\big[(1+2(N-2)/k)\big]^{k(k-1)/2}.
\end{equation}

Choose an orthonormal eigenbasis $(\b u_\alpha)$ of $H$. As $H\overset d=\mathcal U H\mathcal U^*$, there is an orthonormal eigenbasis $(\b v_\alpha)$ of $\cal UH\cal U^*$ such that $(\b u_\alpha) \overset{d}{=}(\b v_{\alpha})$. By Lemma \ref{multiplicity}, every $\b v_\alpha$ is a linear combination of at most two $\cal U\b u_{i}$. By Cauchy Schwarz and taking maximum over $\alpha$, we have $\max_{\alpha}|\langle \b w,\b v_\alpha\rangle |\leq \sqrt2\max_{\alpha}|\langle\b w,\mathcal U\b u_{\alpha}\rangle|$. Applying \eqref{lll} with $\b u=\b u_{\alpha}$ and Markov's inequality at the scale $(2L)^{-1/2}\e^t$ yields 
\[
\begin{aligned}
	&\mathbb P\Big(\max_{\substack{\alpha}}|\langle\b w,\b u_\alpha\rangle|\geq L^{-1/2}\e^t\Big)=\mathbb P\Big(\max_{\substack{\alpha}}|\langle\b w,\b v_\alpha\rangle|\geq L^{-1/2}\e^t\Big)\leq  \bb P\Big(\max_{\alpha}|\langle\b w,\mathcal U\b u_{\alpha}\rangle|\geq (2L)^{-1/2}\e^t\Big)\\	
	\leq &\,\sum_{\alpha}  P\Big(|\langle\b w,\mathcal U\b u_{\alpha}\rangle|\geq (2L)^{-1/2}\e^t\Big) \leq L4^k\big[(1+2(N-2)/k)\big]^{k(k-1)/2}\e^{-2kt}.
\end{aligned}
\]
Set $k=4c\sqrt{N/\log N}+O(1)$ and $t=c\sqrt{N\log N}$. Note that
\[
\log L+k\log4+\frac{k(k-1)}2\log(1+2(N-2)/k)-2kt=\Big(-4c^2+\frac{\log 2}{2}+o(1)\Big)N\,,
\]
and thus we get \eqref{thmbound} as desired.
\end{proof}

\end{document}